\RequirePackage{fix-cm}
\documentclass[smallextended]{svjour3}       
\smartqed  
\usepackage{graphicx}
\usepackage{amsmath}
\usepackage{amssymb}
\usepackage{mathtools}

\begin{document}

\title{Order-to-topology continuous operators
}


\author{Kazem Haghnejad Azar}


\institute{K. Haghnejad Azar \at
           Department of Mathematics\\
           University of Mohaghegh Ardabili\\
           Ardabil, Iran\\
           \email{haghnejad@uma.ac.ir}}

\date{Received: date / Accepted: date}

\maketitle

\begin{abstract}
An operator $T$ from vector lattice $E$ into vector topology $(F,\tau)$ is said to be order-to-topology continuous whenever $x_\alpha\xrightarrow{o}0$ implies $Tx_\alpha\xrightarrow{\tau}0$ for each $(x_\alpha)_\alpha\subset E$. The collection of all order-to-topology continuous operators will be denoted by $L_{o\tau}(E,F)$. In  this paper, we will study some properties of this new classification of operators. We will investigate the relationships between order-to-topology continuous operators and others classes of operators such as order continuous, order weakly compact and $b$-weakly compact operators.
\end{abstract}

\section{Introduction}
In locally solid vector lattice, topologies for which order convergence implies topological convergence are very useful. They are known as order continuous topologies. A linear topology $\tau$ on a vector lattice is said to be order continuous whenever $x_\alpha\xrightarrow{o}0$  implies $x_\alpha\xrightarrow{\tau}0$. 
In normed vector lattice, it is also favourite to us, when order convergence is norm convergent. A normed lattice $E$ has order continuous norm if $\| x_\alpha\|\rightarrow 0$ for every decreasing net $(x_\alpha)_\alpha$ with $\inf_\alpha x_\alpha=0$. Let $E$ be a vector lattice and $(F,\tau)$ be a vector topology. In this manuscript, we will investigate on operators $T:E\rightarrow F$ which carrier every order convergence net into topological convergence. To state our results, we need to fix some notation and recall some definitions.
A net $(x_{\alpha})_{\alpha \in A}$ in a vector lattice $ E $ is said to be strongly order convergent to $x\in E$ if there is a net $(z_{\beta})_{\beta \in B} $ in $ E $ such that $ z_{\beta} \downarrow 0 $ and for every $ \beta \in B$, there exists $\alpha_{0} \in A$ such that $ | x_{\alpha} - x |\leq z_{\beta}$ whenever $ \alpha \geq \alpha_{0}$. For short, we will denote this convergence by $ x_{\alpha} \xrightarrow{so} x $ and write that $ x_{\alpha} $ is $so$-convergent to $x$. Obviusely every order convergence net in a vector lattice is strongly order convergent, but converse not holds and for Dedekind complete vector lattice both definitions are the same, for detile see \cite{1b}. A net $ (x_{\alpha})_{\alpha}$ in vector lattice $ E $ is unbounded order convergent to $ x \in E $ if $ | x_{\alpha} - x | \wedge u \xrightarrow{so} 0$ for all $ u \in E^{+} $. We denote this convergence by $ x_{\alpha} \xrightarrow{uo}x $ and write that $ x_{\alpha} $ $uo$-convergent to $ x $. It is clear that for order bounded nets, $uo$-convergence is equivalent to $so$-convergence. In \cite{7}, Wickstead characterized the spaces in which $w$-convergence of nets implies $uo$-convergence and vice versa and in \cite{5g1}, characterized the spaces $E$ such that in its dual space $ E^{\prime} $, $uo$-convergence implies $w^{*}$-convergence and vice versa. A Banach lattice $E$ is said to be an $AM$-space if for each $x,y\in E$ such that $|x|\wedge |y|=0$, we have $\|x+y\|= max \{\|x\|, \|y\|\}$.  A Banach lattice $E$ is said to be $KB$-space whenever each increasing norm bounded sequence of $E^+$ is norm convergent. An operator $T: E\rightarrow F$ between two vector lattices  is positive if $T(x)\geq 0$ in $F$ whenever $x\geq 0$ in $E$. Note that each positive linear mapping on a Banach lattice is continuous. 
In this manuscript $L_b(E,F)$ is the all of bounded operators and 
the collection of all order continuous operators of $L_b(E,F)$ will be denoted by $L_n(E,F)$; the subscript $n$ is justified by the fact that the order continuous operators are also known as normal operators. That is, 
$$L_n(E,F) :=\{T \in L_b(E,F): T~ \text{is~ order~ continuous}\}.$$
Similarly, $L_c(E,F)$ will denote the collection of all order bounded operators from $E$ to $F$ that are $\sigma-$order continuous. 
An operator $T$ from a Banach space $X$ into a Banach space $Y$ is compact (resp. weakly compact) if $\overline{{T(B _ X)}}$  is compact (resp. weakly compact) where $B _ X$ is the closed unit ball of $X$.   A continuous operator from Banach lattice $E$ into Banach space $X$ is called $M$-weakly compact if $\lim \Vert Tx_n\Vert=0$ holds for every norm bounded disjoint sequence $(x_n)_n$ of $E$. A subset $A$ of a vector lattice $E$ is called $b$-order bounded in $E$ if it is order bounded in $E^{\sim\sim}$.  An operator $T:E\rightarrow X$,  mapping each $b$-order bounded subset of $E$ into a relatively weakly compact subset of $X$ is called a $b$-weakly compact operator, see \cite{3}.
An operator $T:E\rightarrow X$ from vector lattice into normed space is called interval-bounded if the image of every order interval is norm bounded. For every interval-bounded linear operator $T:E\rightarrow X$, set
\begin{equation*}
q_T(x)=\sup\{\Vert Ty\Vert:~\vert y\vert\leq \vert x\vert\},
\end{equation*}
be the absolute monotone seminorm induced by $T$ where $x\in E$.
 For terminology concerning Banach lattice theory and positive operators, we refer the reader to the excellent book of \cite{1}.\\

\section{Main results}

Let $E$ be a vector lattice and $F$ be a  vector topology  with topology $\tau$.
An operator $T$ from $E$ into  $F$ is said to be order-to-topology continuous whenever $x_\alpha\xrightarrow{o}0$ implies $Tx_\alpha\xrightarrow{\tau}0$ for each $(x_\alpha)_\alpha\subset E$. For each sequence $(x_n)\subset E$, if $x_n\xrightarrow{o}0$ implies $Tx_n\xrightarrow{\tau}0$, then  $T$ is called  $\sigma$-order-to-topology continuous operator. 
The collection of all order-to-topology continuous operators will be denoted by $L_{o\tau}(E,F)$;
the subscript $o\tau$ is justified by the fact that the order-to-topology continuous operators;
that is, 
\begin{equation*}
L_{o\tau}(E,F)=\{T\in L(E,F):~T~\text{is order-to-topology continuous }\}.
\end{equation*}
Similarly, $L^\sigma_{o\tau}(E,F)$ will be denote the collection of all $\sigma$-order-to-topology continuous operators, that is, 
\begin{equation*}
L^\sigma_{o\tau}(E,F)=\{T\in L(E,F):~T~\text{is} ~\sigma-\text{order-to-topology continuous }\}.
\end{equation*}

For a normed space $F$, we write $L_{on}(E,F)$ and $L_{ow}(E,F)$ for collection of order-to-norm topology continuous operators and order-to-weak topology continuous operators, respectively.
$L^\sigma_{on}(E,F)$ and $L^\sigma_{ow}(E,F)$ have similar definitions. Clearly $L^\sigma_{on}(E,F)$ is a subspace of $L^\sigma_{ow}(E,F)$ and if $F$ has the Schur property, then $L^\sigma_{on}(E,F)=L^\sigma_{ow}(E,F)$. Let $T$ be an order-to-norm topology continuous operators from a vector lattice $E$ into a normed vector lattice $F$ and $0\leq S\leq T$ where $S\in L(E,F)$. Then observe that $S$  is an order-to-norm topology continuous operator. It is clear that for a locally solid vector lattice $E$, if $E$ has order continuous topology, then every continuous operator from $E$ into a vector topology $F$ is order-to-topology continuous. If an operator $T:E\rightarrow F$ from a Banach lattice into a normed vector lattice is positive (and so continuous), in general, $T$ is not order-to-norm continuous as shown in the following examples.

\begin{enumerate}
\item Consider the operator $T:\ell^\infty\rightarrow\ell^\infty$ defined by $T((a_n)_n)=(ca_n)_n$ where $c>0$ and $\ell^\infty$ equipped with norm $\|.\|_\infty$. We observe that $T$ is positive (and so continuous). Now let $x_n=(0,0,...0,1,1,..)$, which have first $n$ zero terms and all the other terms equal one. It is clear that $x_n \downarrow 0$ and $\Vert x_n\Vert=1$. It follows that $T$ is not order-to-norm continuous. 

\item If $E=C([0,1])$ and $f_n:~f_n(t)=t^n$ for all $t\in [0,1]$, then $f_n \downarrow 0$ as $n\rightarrow \infty$ and $\Vert f_n\Vert=1$. Now we define an operator $T:C([0,1])\rightarrow C([0,1])$ with $T(f)=f$  which is positive (and so continuous), but $T$ is not order-to-norm continuous. 
\end{enumerate}

The following example that is inspired from example 3.3 in \cite{alha2018} shows that an order-to-norm continuous operator need not to be continuous.
\begin{exa}
Let 
$E=\left\{ (a_n)\in c_0 \mid (na_n)\in c_0\right\}$. It is easy to see that $E$ is a normed vector sublattice of Banach lattice $c_0$. We define operator $T$ from $E$ into $c_0$ as follows
$$T(a_1, a_2, a_3,\cdots)=(a_1, 2a_2,3a_3,\cdots). $$
Clearly, $T$ is not continuous. Now, let $(x_\alpha)$ be a net in $E$ such that $x_\alpha\xrightarrow{o}0$.
So there exists some $(y_\alpha)\subset E$ such that $|x_\alpha|\leq y_\alpha\downarrow 0$ for all $\alpha$.
We have $|Tx_\alpha |\leq T|x_\alpha|\leq Ty_\alpha$. On the other hand, if we denote $n$-th term of $y_\alpha$ by $y_\alpha(n)$, then we have
\begin{align*}
y_\alpha\downarrow 0 & \iff y_\alpha(n)\downarrow 0 \qquad (\forall n\in\mathbb{N}) \\
& \iff ny_\alpha(n)\downarrow 0 \qquad (\forall n\in\mathbb{N})\\
& \iff Ty_\alpha \downarrow 0.
\end{align*}
Therefore, $|Tx_\alpha |\leq Ty_\alpha \downarrow 0$. Since $c_0$ has order continuous norm, $\|Ty_\alpha\|\to 0$. Now, it follows from $\|Tx_\alpha\|\leq\|Ty_\alpha\|\to 0$ that $T$ is order-to-norm continuous.
\end{exa}
 
As next example, we  observe that a  $\sigma$-order-to-topology continuous operators need not be order-to-topology continuous operators, see Example 1.55 from \cite{1}.

 \begin{exa}
 Let $E$ be the vector space of all Lebesgue integrable real valued function defined on $[0,1]$. Define the operator $T:E\rightarrow \mathbb{R}$ by 
\begin{equation*} 
 T(f)=\int_0^\infty f(x)dx.
\end{equation*} 
  If $f_n\downarrow 0$, clearly $ \vert Tf_n \vert\rightarrow 0$. Thus $T\in L_{on}^\sigma(E,\mathbb{R} )$. Now, let $\Lambda$ denotes the collection of all finite subsets of $[0,1]$. Then $\{\chi_\alpha~:~\alpha\in \Lambda\}\subseteq E$ (where $\chi_\alpha$ is the characteristic function of $\alpha$) satisfies   
$1-\chi_\alpha\downarrow 0$. On the other hand, observe that $T(1-\chi_\alpha )=1$ which shows that $T\notin L_{on}(E,\mathbb{R})$.  
 \end{exa}

\begin{prop} \label{2.5}
Let $T:E\rightarrow F$ be an order bounded operator between two normed vector lattices. By one of following conditions $T^-, T^+, T \text{and } \vert T\vert\in L_{on}(E,F)$.
\begin{enumerate}
\item $F$ has order continuous norm and $T\in L_n(E,F)$
\item $E$ is a Banach lattice with order continuous norm and $F$ is Dedekind complete.
\end{enumerate}
\end{prop}
\begin{proof}
\begin{enumerate}
\item Since $F$ has order continuous norm by Corollary 4.10 from \cite{1}, $F$ is Dedekind complete. It follows that $T^-$ and $T^+$ exist. Now, let $(x_\alpha)\subset E^+$ and $x_\alpha\downarrow 0$. Since $T$ is order continuous, it follows $T^-x_\alpha\downarrow 0$ and $T^+x_\alpha\downarrow 0$. As $F$ has order continuous norm, we have $\|T^-x_\alpha\|\rightarrow 0$ and $\|T^+x_\alpha\|\rightarrow 0$, and so by using inequalities $\|Tx_\alpha\|\leq\|T^+x_\alpha\|+\|T^-x_\alpha\|$ and $\| \vert T \vert x_\alpha\|\leq\|T^+x_\alpha\|+\|T^-x_\alpha\|$ proof holds.
\item Let $(x_\alpha)\subset E^+$ and $x_\alpha\downarrow 0$. By Theorem 4.3, from \cite{1}, $T^+$ is norm continuous. Consequently  $\|T^+x_\alpha\|\leq \|T^+\|\|x_\alpha\|\rightarrow 0$, since $E$ has order continuous norm. Similarly, $\|T^-x_\alpha\|\rightarrow 0$, and so proof follows immediately. 
\end{enumerate}
\end{proof}

The next example shows that order continuity of $T$ in above proposition, can not in general be dropped.

\begin{exa}
Let $E$ be the collection of real functions on $[0,1]$ that are continuous except on a finite subset of $[0,1]$. Let $T:E\rightarrow L^1([0,1])$ be an operator with $T(f)=f$. We know that $L^1([0,1])$ has order continuous norm and $T$ is order bounded. Let $Q\cap [0,1]=\{r_1,~r_2,~r_3,\ldots\}$ and $F_n=\{r_1,~r_2,\ldots,r_n\}$. We define the sequence of functions in $E$ as follows:
\begin{eqnarray*}
    f_n(x)=\left\lbrace \begin{array}{lc}
0\quad \text{if}~~~x\in F_n\\
1\quad\text{if}~~~x\in [0,1]\setminus F_n
     \end{array}\right.    
\end{eqnarray*}
Clearly $f_n\downarrow 0$, but $\|Tf_n\|=1$. Thus  $T$ is not $\sigma$-order-to-topology continuous operator. 
\end{exa}

Recall that an operator $T:E\rightarrow F$ between two vector lattices is said to  preserve disjointness whenever $x\bot y$ in $E$ implies $Tx\bot Ty$ in $F$. By using Theorem 3.1.4, from \cite{6}, we have the following result.

\begin{thm} \label{2.7}
Let $E$ and $F$ be normed vector lattice and $T$ be an order bounded disjointness operator from $E$ into $F$.  Then $T\in L_{on}(E,F)$ if and only if $T^-, T^+ \text{and } \vert T\vert\in L_{on}(E,F)$. 
\end{thm}

By notice to preceding theorem,   for $1\leq p, q<\infty$ with $\frac{1}{p}+ \frac{1}{q}=1$, we have the following assertions.
\begin{enumerate}
\item  $L_{on}(\ell_p, \mathbb{R})=\ell_q$.
\item  $L_{on}(L_p([0,1]), \mathbb{R})=L_q([0,1])$.
\end{enumerate}
\begin{prob}
\begin{enumerate}
\item Dose the modulus of an operator $T$ from vector lattice $E$ into normed vector lattice $F$ exists and $\vert T\vert\in L_{on}(E, F)$ or $\vert T\vert\in L_{ow}(E, F)$ whenever $T\in L_{on}(E, F)$ or $T\in L_{ow}(E, F)$, respectively?
\item Is $ L_{on}(E, F)$ a band in $ L_{ow}(E, F)$?
\end{enumerate}
\end{prob}

%

\begin{lem} \label{2.7}
Let $E$ be a vector lattice and $F$ be a Dedekind complete normed vector lattice. If $T$ is positive then   $T\in L_{on}(E,F)$ if and only if $x_\alpha\downarrow 0$ implies $\|Tx_\alpha\|\rightarrow 0$. 
\end{lem}
\begin{proof}
We just prove one side. The other side is clear. Let $(x_\alpha)$ be a net in $E$ such that $x_\alpha\xrightarrow{o}0$.
So there exists some $(y_\alpha)\subset E$ such that $|x_\alpha|\leq y_\alpha\downarrow 0$ for all $\alpha$.
So $|Tx_\alpha |\leq T|x_\alpha|\leq Ty_\alpha$ for all $\alpha$. Thus, $\|Tx_\alpha \|\leq \|Ty_\alpha \|$ for all $\alpha$. By our hypothesis $\|Ty_\alpha \| \to 0$. Hence $\|Tx_\alpha \|\to 0$; that is, $T\in L_{on}(E,F)$.
\end{proof}

\begin{thm}\label{t:2.8}
Let $E$ be a $\sigma$-Dedekind complete vector lattice and $F$ be a normed vector lattice. An order bounded operator $T:E\rightarrow F$ is $\sigma-$order-to-norm continuous  if and only if $0<x_n\uparrow \leq x$ in $E$ implies 
$(Tx_n)$ is norm convergent to $T(\sup_nx_n)$.
\end{thm}
\begin{proof}
Let $T$ be a $\sigma$-order-to-norm continuous operator and let $(x_n)\subset E^+$ and $x\in E$ where $0<x_n\uparrow \leq x$.  We set $\sup x_n=y$. It follows $(y-x_n)\downarrow 0$, and so $\|T(y-x_n)\|\rightarrow 0$. Thus $(T x_n)$ is norm convergent to $Ty$. \\ 
Conversely, by preceding Lemma, without lose generality assume that $T$ is a positive operator. Let $(x_n)\subset E^+$ with $x_n\downarrow 0$. We set $y_n=x_1-x_n$. Observe that $0\leq y_n\uparrow \leq x_1$. Thus $(Ty_n)$ is norm convergent to $Tx_1$, since $\sup_ny_n=x_1$, which follows that $(Tx_n)$ is norm convergent.
Therefore, $T\in L^\sigma_{on}(E,F)$.
\end{proof}

\begin{cor}\label{2.8}
Let $E$ be a $\sigma$-Dedekind complete vector lattice and $F$ be a normed vector lattice. If $T$ is interval-bounded, then  $T\in L^\sigma_{on}(E,F)$ if and only if $T$ is order weakly compact. 
\end{cor}
\begin{proof}
By Theorem 3.4.4 from \cite{6}, $T$ is order weakly compact operator if and only if $(Tx_n)$ is convergent for every order bounded increasing sequence $(x_n)$ in $E^+$. So by the Theorem \ref{t:2.8} proof holds.
\end{proof}

\begin{thm}\label{2.12}
Let $E$ and $F$ be normed vector lattice with $F$ Dedekind complete. If $E$ or $F$ has order continuous norm, then $L_n(E,F)$ is a band in $L_{on}(E,F)$.
\end{thm}
\begin{proof}
Let $T\in L_n(E,F)$. It follws $T^-,~T^+\in L_n(E,F)$, and so by using Theorem 4.3, \cite{1}, $T^-$ and $T^+$ are norm continuous. As $E$ has order continuous norm,   $T^-,~T^+\in L_{on}(E,F)$, and so $T\in L_{on}(E,F)$.
Thus $L_n(E,F)$ is a subspace of $L_{on}(E,F)$.
Now let $\vert S\vert\leq\vert T\vert$ where $S\in L_{n}(E,F)$ and  $T\in L_{on}(E,F)$. Then for each $x\in E$, we have $\vert Sx\vert\leq \vert S \vert(\vert x\vert )\leq \vert T \vert (\vert x\vert )$. It follows that 
\begin{equation*}
\|Sx\|\leq \| \vert Sx\vert \|\leq \| \vert S \vert(\vert x\vert ) \|\leq \|     \vert T \vert (\vert x\vert )\|\leq \| T \| \|x\|          
\end{equation*}
The above inequalities shows that $S$ is norm-to-norm continuous.  Let $x_\alpha \downarrow 0$. As $E$ has order continuous norm,  it follows that $\| x_\alpha\|\rightarrow 0$, and so $\| Sx_\alpha\|\rightarrow 0$. Than $S\in L_{on}(E,F)$. Thus $L_n(E,F)$ is an ideal in $L_{on}(E,F)$. To see that the ideal $L_n(E,F)$ is a band, let $0\leq T_\lambda\uparrow T$ in $L_{on}(E,F)$.
As $T$ is positive, $T$ is norm continuous. Since $E$ has order continuous norm, proof follows immediately.\\ 
Now if $F$ has order continuous norm, we have similar argument.

\end{proof}

\begin{thm}\label{2.13}
Let $E$ be a vector lattice and  $F$ a normed vector lattice  which every norm null net in $F$ is order bounded. Then
\begin{enumerate}
\item $L_{on}(E,F)$ is an ideal in $L_{b}(E,F)$.
\item If $L_{on}(E,F)$ is order dense in $L_b(E,F)$, then $L_b(E,F)=L_n(E,F)$, and if 
 $E$ has also norm continuous and $F$ is Dedekind complete, then $L_{on}(E,F)=L_n(E,F)=L_b(E,F)$.
\end{enumerate}
\end{thm}
\begin{proof}
\begin{enumerate}

\item Let $x\in E^+$ and consider  a net $(x_\alpha)_\alpha$, where $x_\alpha =x-\alpha$ for each $\alpha\in [0,x]$.  It follows that  $x_\alpha\downarrow 0$.   If $T\in L_{on}(E,F)$, then $(Tx_\alpha)$ is null-norm convergent in $F$. It follows that $(Tx_\alpha)_\alpha$ is norm bounded in $F$. By assumption, $(Tx_\alpha)_\alpha$ is order bounded, which shows that $L_{on}(E,F)$ is a subspace of $L_{b}(E,F)$. Now let $\vert S\vert\leq\vert T\vert$ where $T\in L_{b}(E,F)$ and  $S\in L_{on}(E,F)$. Then for each $x\in E$, we have $\vert Sx\vert\leq \vert S \vert(\vert x\vert )\leq \vert T \vert (\vert x\vert )$. It follows that $S$ is order bounded, and so $L_{on}(E,F)$ is an ideal in $L_{b}(E,F)$. 
\item Assume that $T\in L_b(E,F)$.  Since $L_{on}(E,F)$ is order dense in $L_b(E,F)$, there is a $(T_\lambda)_\lambda\subset L_{on}(E,F)$  such that   $0\leq T_\lambda\uparrow T$ in $L_{b}(E,F)$. Let $x_\alpha \downarrow 0$ in $E$. Then for each fixed index $\lambda$ we have  
$\Vert T_\lambda (x_\alpha)\Vert\rightarrow 0$ and $T_\lambda(x_\alpha)\downarrow$, and so by Theorem 5.6, from \cite{2}, we have $T_\lambda(x_\alpha)\downarrow 0$. It follows that $(T_\lambda)_{\lambda}\subset L_n(E,F)$. Then by using Theorem 1.57 in \cite{1}, we have $T\in L_n(E,F)$. It follows  that  $L_b(E,F)=L_n(E,F)$. Now if $E$ has order continuous, then proof follows immediately from Theorem 5.6 from \cite{2}.  
\end{enumerate}
\end{proof}

 As notice to Theorem \ref{2.12} and Theorem \ref{2.13},  we have the following consequence
 \begin{cor}\label{2.8}
Let $E$   be a normed vector lattice. Then
\begin{enumerate}
\item $L_n(c_0,\ell^\infty)$ is a band in $L_{on}(c_0,\ell^\infty)$.
\item $L_{on}(c_0,\ell^\infty)$ is an ideal in $L_b(c_0,\ell^\infty)$.
\item $L_{on}(E,\mathbb{R})$ and $L_{on}(E,\ell^\infty)$ are  ideal in $ E^\prime$  and $ L_b(E,\ell^\infty)$, respectively where $E^\prime$ is norm dual of $E$.
\item if $E$ has order continuous norm, then $E_n^\thicksim =E^\thicksim=L_{on}(E,\mathbb{R}) $. 

\end{enumerate}
\end{cor}
%

\begin{thm} \label{2.14}
Let $E$ and $F$ be a vector normed lattices and $F$ Dedekind complete. Let $T\in L_b(E,F)$. Then the following assertions are equivalent. 
\begin{enumerate}
\item $E$ is Dedekind $\sigma$-complete and $x_n\downarrow 0$ in $E$ implies $\Vert Tx_n\Vert\rightarrow 0$.
\item If $0<x_n\uparrow \leq x$ holds in $E$, then  $(Tx_n)$ is norm convergent.
\item $T\in L_{on}(E,F)$
\end{enumerate}
\end{thm}
\begin{proof}
In the following, without lose generality, we assume that $T$ is positive operator.\\
$(1)\Rightarrow (2)$ Let $(x_n)\subset E$ and  $0<x_n\uparrow \leq x$ holds in $E$. Set $\sup x_n=y$. It follows $y-x_n\downarrow 0$. By hypothesis, we have
$\Vert T(y-x_n)\Vert \to 0$. It follows that $(Tx_n)$ is norm
 convergent.\\
$(1)\Rightarrow (3)$  By lemma \ref{2.7}, it is enough to prove that $Tx_\alpha$ is norm convergent to zero in $F$ whenever $x_\alpha\downarrow 0$ in $E$. Let $(x_\alpha)_\alpha\subset E^+$ with  $x_\alpha \downarrow 0$. If $(Tx_\alpha)_\alpha$ is not norm convergent, then it is not a norm Cauchy net. Thus there exists some $\epsilon>0$ and a sequence $(\alpha_n)$ of indices with
$ \alpha_n\uparrow$,
such that $\| Tx_{\alpha_n}-Tx_{\alpha_{n+1}}\|>\epsilon$ for all 
$n$. On the other hand, since $E$ is Dedekind $\sigma$-complete 
there is $x\in E$ such that $x_{\alpha_n}\uparrow x$. It follows 
$x-x_{\alpha_n}\downarrow 0$. By hypothesis, we see that 
$Tx_{\alpha_n}$ is norm Cauchy sequence, which contradicts with
$\|Tx_{\alpha_n}-Tx_{\alpha_{n+1}}\|>\epsilon$. Thus 
$(Tx_\alpha)_\alpha$ is  norm convergent to some point $y\in F$. 
By Theorem 5.6, \cite{2}, we see that $y=0$,  and so $T\in L_{on}(E,F)$.\\
$(3)\Rightarrow (1)$ Obviously.\\
$(2)\Rightarrow (1)$ Let $x_n\downarrow 0$. Then $0\leq (x_1-x_n)\uparrow x_1$ holds in $E$. Then we have $0\leq T(x_1-x_n)\uparrow Tx_1$. By assumption the sequence $\{T(x_1-x_n))$ is norm convergent, and so  Theorem 4.9 from \cite{2}, follows that $\Vert Tx_n\Vert\rightarrow 0$.  
\end{proof}

\begin{cor}\label{2.9}
Assume that $T:E\rightarrow F$ is an order  bounded operator from  $\sigma$-Dedekind complete vector lattice into  normed vector lattice. 
\begin{enumerate}
\item $T\in L^\sigma_{on}(E,F)$ if and only if $T$ is order weakly compact. 
\item If   $T\in L^\sigma_{on}(E,F)$, then $T$ is $M$-weakly compact.
\end{enumerate}
 \end{cor}
\begin{proof}
\begin{enumerate}
\item By using Theorem 3.4.4 from \cite{6} and Theorem \ref{2.14} proof follows.
\item By  Theorem 5.57, \cite{1}, proof follows.
\end{enumerate}
\end{proof}

\begin{cor} \label{2.y}
Let $T: E\rightarrow X$ be an order bounded bounded operator from $\sigma$-Dedekind complete vector lattice into Banach space. Then the following assertions are equivalent. 
\begin{enumerate}
\item $q_T(x_n)\rightarrow 0$ as $n\rightarrow 0$ for every order bounded disjoint sequence.
\item $\vert Tx_n\vert\rightarrow 0$ as $n\rightarrow 0$ for every order bounded disjoint sequence.
\item If $0<x_n\uparrow \leq x$ holds in $E$, then  $(Tx_n)$ is norm convergent.
\item $T$ is order-weakly compact.
\item $T\in L^\sigma_{on}(E,F)$.
\end{enumerate}
\end{cor}
\begin{proof}
By using Theorem 3.4.4 from \cite{6} and Corollary \ref{2.8} proof holds.
\end{proof}

 Alpay-Altin-Tonyali introduced the class of $b$-weakly compact operators for vector lattices having separating order duals \cite{3}. In  \cite{4}, Alpay and Altin proved that a continuous operator $T$ from a Banach lattice $E$ into a Banach space $X$ is $b$-weakly compact if and only if $(Tx_n)_n$ is norm convergent for each $b$-order bounded increasing sequence $(x_n)_n$  in  $E^+$ if and only if  $(Tx_n)_n$ is norm convergent to zero for each $b$-order bounded disjoint sequence $(x_n)_n$ in $E^+$. In \cite{5} authors proved that an operator $T$ from a Banach lattice $E$ into a Banach space $X$ is $b$-weakly compact if and only if $(Tx_n)_n$  is norm convergent for every positive increasing sequence $(x_n)_n$ of the closed unit ball  $B_E$ of $E$.
Now in the following we study the relationships between two classifications of operators, $b-$weakly compact and $\sigma-$order-to-norm continuous operators.  

\begin{thm} \label{2}
Let $E$  and $F$ be normed vector lattice, and let $T$ be an operator from $E$ into $F$. Then the following assertions hold 
\begin{enumerate}
\item  If $F$ is Dedekind $\sigma$-complete, then each positive $b$-weakly compact operator is $\sigma$-order-to-norm continuous operator. 
\item If $E$ has a order unit with Dedekind $\sigma$-complete, then each $\sigma$-order-to-norm continuous operator is $b$-weakly compact operator
\end{enumerate}
\end{thm}
\begin{proof}
\begin{enumerate}
\item Let $T$ be a $b$-weakly compact operator and $(x_n)_n\subset E$ with $x_n \downarrow 0$. Set $y_n=x_1-x_n$. Then $y_n\uparrow$ and $\sup_n\|y_n\|\leq \Vert x_1\Vert$. It follows that $(T(y_n))_n$ is norm convergent and $Ty_n\uparrow$. By using Theorem 3.46 from \cite{1}, $(T(y_n))_n$ norm convergence to $Tx_1$. Thus $(T(x_n))_n$ is norm convergent to $0$.
\item Let $(x_n)_n$ be an increasing positive sequence in $E$ with $\sup_n\Vert x_n\Vert<\infty$. Since   $E$ has  order unit, $(x_n)_n$ is order bounded, and so there is a positive element $x\in E$ such that $x_n \uparrow \leq x$. Then there exists $y\in E$ such that $\sup_n x_n=y$. It follows that $0\leq y-x_n\downarrow 0$. As $T\in L_{on}^\sigma (E,F)$, we have $\Vert T(y-x_n)\Vert\rightarrow 0$ which implies $(T(x_n))_n$  is norm convergence, and so $T$ is $b$-weakly compact.
\end{enumerate}
\end{proof}
 It is clear that the identity operator $ I:l^1 \rightarrow l^1 $ is an order-to-norm continuous operator, but its adjoint $ I:l^\infty \rightarrow l^\infty $ is not order-to-norm continuous. Note that the identity operator $ I : l^\infty \rightarrow l^\infty $ is not order-to-norm continuous, while its adjoint is order-to-norm continuous. The following results, give a sufficient and necessary condition for which the order-to-norm continuity of an operator implies the order-to-norm continuity of its adjoint and reverse.
\begin{thm}
Let $E$ and $ F$ be two Banach lattices with $E$ Dedekind complete. Then the following conditions are equivalent.
\begin{enumerate}
	\item Each continuous operator from $ F^\prime $ into $ E^\prime$ is order-to-norm continuous.
	\item If a continuous operator $ T:E\rightarrow F $ is an order-to-norm continuous operator, then its adjoint $ T^\prime $ is order-to-norm continuous.
	\item $ F^\prime $ is a $KB$-space.
	\end{enumerate}	
\begin{proof}
	$1 \Rightarrow 2$ Obvious.
	
\noindent $2 \Rightarrow 3$ Assume by the way of contradiction, that $ F^\prime $ is not a $KB$-space.  By Lemma 2.1 of \cite{dualbweak} there is a positive order bounded disjoint sequence $(g_n)$ of $F^\prime$ satisfying $\|g_n\| =1 $. Since $\|g_n\|=\sup\{g_n(y): 0\leq y \in F \text{ and } \|y\|\leq 1 \}$ holds for all $n$,  we choose $y_n \in F^+$ with $\|y_n \|=1$ and $g_n(y_n) \geq \frac{1}{2}$.  Now, we consider a positive operator $ T:l^1 \rightarrow F $ defined by 
	\begin{equation*}
		T((a_n))= \sum_{n=1}^{\infty} a_ny_n \text{ for all } (a_n)\in l^1.
		\end{equation*}
	Clearly, $ T $ is well defined. Since $ T $ is positive, therefore $T$ is continuous and hence $ T $ is  order-to-norm continuous. But its adjoint $ T^\prime : F^\prime \rightarrow l^\infty $ defined by 
	\begin{equation*}
		T^\prime(h)= (h(y_n))  \text{ for all } h\in F^\prime .
		\end{equation*}
As follows, we shows that $T^\prime$	is not order-to-norm continuous. Since $(g_n)$ is disjoint, by using Corollary 3.6 from \cite{5g} obvious $g_n \xrightarrow{uo} 0$ in $F^\prime$. Now since $ (g_n)$ is order bounded, it is clear that $g_n \xrightarrow{o} 0 $ in $ F^\prime$. On the other hands, we have
	\begin{equation*}
		\|T^\prime (g_n)\|= \|(g_n(y_m))_m\| \geq\| g_n(y_n)\|\geq \frac{1}{2} \ holds \ for \ all \ n.
	\end{equation*}
Therefore $T^\prime$ is not order-to-norm continuous.

\noindent $3\Rightarrow 1$ Obvious.	
	\end{proof}
\end{thm}

\begin{thm}
	Let $ E $ be a vector lattice with Dedekind complete and property (b), and $F$ be a Banach lattice. Then the following statements are equivalent.
	\begin{enumerate}
		\item Each continuous operator from $ E $ into $ F $ is order-to-norm continuous.
		\item Each continuous operator $ T : E\rightarrow F $ is order-to-norm continuous whenever its adjoint $ T^\prime $ is order-to-norm continuous.
		\item $ E $ is a $KB$-space
	\end{enumerate}
\begin{proof}
	$ 1 \Rightarrow 2 $ Obvious.
	
\noindent $ 2 \Rightarrow 3 $ Let $ E $ is not $KB$-space. To finish the proof, we have to construct an operator $ T:E \rightarrow F $ such that $ T $ is not order-to-norm continuous but its adjoint $ T^\prime $ is order-to-norm continuous. 
	
	Since $ E $ is not $KB$-space then, it follows from Lemma 2.1 of \cite{dualbweak} that $ E^+$ contains a b-order bounded disjoint sequence $ (x_n) $ satisfying $\| x_n \| =1$ for all $n$. So by Lemma 3.4 of \cite{dualbweak}, there exists a positive disjoint sequence $ (g_n)$ of $ E^\prime $ with $\|g_n\| \leq 1 $ such that 
	\begin{equation*}
	g_n (x_n)=1 \ 	for\  all\  n\  and\  g_n(x_m)=0\  for\  all\  n\neq m.
	\end{equation*}
Now we defined the positive operator $ T : E \rightarrow l^\infty $ by 
\begin{equation*}
	T(x) = (g_n(x))_{n=1} ^ \infty\   for\  each\  x\in E.
	\end{equation*}
Note that $ (g_n(x))_{n=1}^\infty \subset R $ and $ \sup_n | g_n (x)| \leq \sup_n \| g_n \| \|x\| \leq \|x\| < \infty$. So, $ T $ is well defined. It is clear that $(x_n) $ is order bounded. By using Corollary 3.6 from \cite{5g}, obvious $ x_n \xrightarrow{uo} 0$, and thus $ x_n \xrightarrow{o} 0 $ in $ E$. Since  $ T(x_n) = e_n $, therefore $T(x_n) \nrightarrow 0 $ in norm. It follows that $ T $ is not order-to-norm continuous. Since the norm of $(l^\infty)^\prime$ is order continuous, it follows that $ T^\prime : (l^\infty)^\prime \rightarrow E^\prime $ is order-to-norm continuous.

\noindent $3 \Rightarrow 1$ Obvious.	
	\end{proof}
\end{thm}




\end{document}